\documentclass[12pt,fleqn,leqno]{amsart}
\usepackage{amsfonts,amssymb}
\usepackage{mathrsfs}
\usepackage[bb=boondox]{mathalfa} 
\usepackage{enumitem}

\usepackage[svgnames]{xcolor}
\usepackage[colorlinks,linkcolor=blue,citecolor=Green]{hyperref}
\usepackage{titlesec}
\usepackage{tikz-cd}
\usetikzlibrary{decorations.pathmorphing}

 \titleformat{\section}[block]
  {\bfseries}
  {\thesection.}
  {.5em}
  {}

\setlist{noitemsep}
\setenumerate{labelindent=\parindent,label=\upshape{(\alph*)}}

\newtheorem{theorem}{Theorem}[section]
\newtheorem{proposition}[theorem]{Proposition}
\newtheorem{corollary}[theorem]{Corollary}

\theoremstyle{definition}
\newtheorem{remark}[theorem]{Remark}

\DeclareMathOperator{\uhr}{\upharpoonright} 
\DeclareMathOperator\sto{\leadsto}
\DeclareMathOperator\card{Card}
\DeclareMathOperator\diam{diam}
\newcommand{\lembed}[1]{\lhook\joinrel\xrightarrow{#1}}

\renewcommand{\emptyset}{\varnothing}
\numberwithin{equation}{section}

\begin{document}

\author{Valentin Gutev} \address{Department of Mathematics, Faculty of
  Science, University of Malta, Msida MSD 2080, Malta}
\email{valentin.gutev@um.edu.mt}



\title{Near-selection theorems for $C$-spaces}
\begin{abstract}
  Haver's near-selection theorem deals with approximate selections of
  Hausdorff continuous CE-valued mappings defined on $\sigma$-compact
  metrizable $C$-spaces. In the present paper, we extend this theorem
  to all paracompact $C$-spaces. The technique developed to achieve
  this generalisation is based on oriented simplicial complexes. This
  approach makes not only a considerable simplification in the
  proof but is also successful in generalising the special case of
  compact metric $C$-spaces to all paracompact finite $C$-spaces.
\end{abstract}

\date{\today}
% \date{}
\maketitle

\section{Introduction}

A space $X$ has property $C$, or is a \emph{$C$-space}, if for any
sequence $\{\mathscr{U} _n:n<\omega\}$ of open covers of $X$ there
exists a sequence $\{\mathscr{V} _n:n<\omega\}$ of pairwise disjoint
families of open subsets of $X$ such that each $\mathscr{V} _n$
refines $\mathscr{U} _n$ and $\bigcup_{n<\omega}\mathscr{V} _n$ is a
cover of $X$. The $C$-space property was originally defined by W.\
Haver~\cite{haver:1974} for compact metric spaces, subsequently Addis
and Gresham~\cite{addis-gresham:78} reformulated Haver's definition
for arbitrary spaces.  It should be remarked that a $C$-space is
paracompact if and only if it is countably paracompact and normal, see
e.g.\ \cite[Proposition 1.3]{MR2352366}. Every finite-dimensional
paracompact space, as well as every countable-dimensional metrizable
space, is a $C$-space~\cite{addis-gresham:78}, but there exists a
compact metric $C$-space which is not
countable-dimensional~\cite{pol:81}.\medskip

For spaces (sets) $X$ and $Y$, we write $\varphi:X\sto Y$ to designate
that $\varphi$ is a \emph{set-valued} (or \emph{multi-valued})
\emph{mapping} from $X$ to the nonempty subsets of $Y$.  For a metric
space $(Y,\rho)$ and $\varepsilon>0$, we will use
$\mathbf{O}_\varepsilon(y)$ for the open $\varepsilon$-ball centred at
a point $y\in Y$; and
$\mathbf{O}_\varepsilon(S)=\bigcup_{q\in S}\mathbf{O}_\varepsilon(q)$,
whenever $S\subset Y$. In these terms, a set-valued mapping
$\varphi:X\sto Y$ is \emph{$\rho$-continuous} (or \emph{Hausdorff
  continuous}) if for every $\varepsilon >0$, each point $x\in X$ has
a neighbourhood $V$ such that
\begin{equation}
  \label{eq:Finite-C-v2:3}
  \varphi(x)\subset \mathbf{O}_\varepsilon(\varphi(p))\ \text{and}\ 
  \varphi(p) \subset \mathbf{O}_\varepsilon(\varphi(x)),\quad
  \text{for every $p\in V$.}
\end{equation}
A map $f:X\to Y$ is an \emph{$\varepsilon$-selection} for a mapping
$\varphi :X\sto Y$ if
$f(x)\in \mathbf{O}_\varepsilon\left(\varphi(x)\right)$, for every
$x\in X$. More generally, for a function
${\varepsilon:X\to (0,+\infty)}$, we say that $f:X\to Y$ is an
\emph{$\varepsilon$-selection} for $\varphi$ if
$f(x)\in \mathbf{O}_{\varepsilon(x)}\left(\varphi(x)\right)$, for
every $x\in X$. Finally, following Haver~\cite{MR503224}, a subset
$S\subset Y$ is called \emph{CE} if for every $\varepsilon>0$ there
exists $\delta>0$ such that $\mathbf{O}_\delta(S)$ is contractible in
$\mathbf{O}_\varepsilon(S)$. The following theorem was proved by Haver
in \cite[Theorem 1]{MR503224}.

\begin{theorem}[\cite{MR503224}]
  \label{theorem-Selections-IDS:1}
  Let $(Y,\rho)$ be a metric space and $X=\bigcup_{n<\omega} X_n$ be a
  metrizable space with each $X_n$ being a compact $C$-space. Then
  each $\rho$-continuous CE-valued mapping $\varphi:X\sto Y$ has a
  continuous $\varepsilon$-selection, for every continuous function
  $\varepsilon:X\to (0,1]$.
\end{theorem}

By a \emph{simplicial complex} we mean a collection $\Sigma$ of
nonempty finite subsets of a set $S$ such that $\tau\in \Sigma$,
whenever $\emptyset\neq \tau\subset \sigma\in \Sigma$. The set
$\bigcup\Sigma $ is the \emph{vertex set} of $\Sigma$, while each
element of $\Sigma$ is called a \emph{simplex}. The
\emph{$k$-skeleton} $\Sigma^k$ of $\Sigma$ ($k\geq 0$) is the
simplicial complex
$\Sigma^{k}=\{\sigma\in \Sigma:\card(\sigma)\leq k+1\}$, where
$\card(\sigma)$ is the cardinality of $\sigma$. In the sequel, for
simplicity, we will identify the vertex set of $\Sigma$ with its
$0$-skeleton $\Sigma^0$.\medskip

The vertex set $\Sigma^0$ of a simplicial complex $\Sigma$ can be
embedded as a linearly independent subset of some linear normed
space. Then to any simplex $\sigma\in \Sigma$ we may associate the
\emph{geometric simplex} $|\sigma|$ which is the convex hull of
$\sigma$. The \emph{geometric realisation} of $\Sigma$ is the set
$|\Sigma|=\bigcup\{|\sigma|:\sigma\in \Sigma\}$.  As a topological
space, we will always consider $|\Sigma|$ endowed with the
\emph{Whitehead topology}~\cite{MR1576810,MR0030759}. This is the
topology in which a subset $U\subset |\Sigma|$ is open if and only if
$U\cap |\sigma|$ is open in $|\sigma|$, for every $\sigma\in \Sigma$.
The following special case of Theorem~\ref{theorem-Selections-IDS:1}
was obtained by Haver in \cite[Proposition 3]{haver:1974}.

\begin{theorem}[\cite{haver:1974}]
  \label{theorem-Selections-IDS-v4:1}
  Let $X$ be a compact metric $C$-space, $(Y,\rho)$ be a metric space
  and $\varphi:X\sto Y$ be a $\rho$-continuous mapping such that for
  every $x\in X$ and $\varepsilon>0$, there exists $\delta>0$ with the
  property that for every finite simplicial complex $\Sigma$, every
  continuous map $g:|\Sigma|\to \mathbf{O}_\delta(\varphi(x))$ is
  null-homotopic in $\mathbf{O}_\varepsilon(\varphi(x))$. Then
  $\varphi$ has a continuous $\varepsilon$-selection, for every
  $\varepsilon>0$.
\end{theorem}

In the present paper, we show that
Theorem~\ref{theorem-Selections-IDS:1} is valid for all paracompact
$C$-spaces, while Theorem~\ref{theorem-Selections-IDS-v4:1} --- for
all paracompact finite $C$-spaces.  In fact, both these
generalisations are obtained with an almost identical proof. Namely,
in the next section, we consider an abstract embedding relation
$S\lembed{*}T$ of subsets $S\subset T\subset Y$ of a space $Y$. The
only requirement we place on this relation is to be \emph{monotone} in
the sense that $A\subset S\lembed{*} T\subset B$ implies
$A\lembed{*} B$.  In this general setting, for a metric space
$(Y,\rho)$, we consider \emph{uniformly $UV^{*}$} subsets defined by
``$\mathbf{O}_\delta(S)\lembed{*} \mathbf{O}_\varepsilon(S)$'', i.e.\
precisely as the properties of the values of $\varphi$ in
Theorems~\ref{theorem-Selections-IDS:1} and
\ref{theorem-Selections-IDS-v4:1}, but now using the relation
$\lembed{*}$. Next, we show that any $\rho$-continuous mapping
$\varphi:X\sto Y$ with uniformly $UV^*$-values, admits a special
sequences $\{\mathscr{U}_n:n<\omega\}$ of open covers of $X$
approximating the values of $\varphi$, see
Theorem~\ref{theorem-Near-sel-v1:1}. This is essentially the only
place where $\rho$-continuity of $\varphi$ is
used. Section~\ref{sec:orient-simpl-compl} contains two constructions
of continuous maps from the geometric realisation of oriented
simplicial complexes, see Theorems~\ref{theorem-Near-sel-v7:1} and
\ref{theorem-Near-sel-v9:1}. Finally, in
Section~\ref{sec:appr-select-c}, we apply
Theorem~\ref{theorem-Near-sel-v7:1} to nerves of covers to obtain with
ease our generalisation of Theorem~\ref{theorem-Selections-IDS:1} (see
Theorem~\ref{theorem-Finite-C-v6:1}). The same approach, but now using
Theorem~\ref{theorem-Near-sel-v9:1}, gives the generalisation of
Theorem~\ref{theorem-Selections-IDS-v4:1} (see
Theorem~\ref{theorem-Finite-C-v5:1}). The last section of the paper
contains two applications of these theorems dealing with
near-selections for Hausdorff continuous mappings whose values are
absolute retracts, see Corollaries~\ref{corollary-Near-sel-vgg:1} and
\ref{corollary-Near-sel-vgg:2}.

\section{Monotone Embedding Relations} 
\label{sec:monot-embedd-relat}

Let $S\lembed{*}T$ be an embedding relation defined on subsets
$S,T\subset Y$ of a space $Y$. In this, we will always assume that
$S\subset T$, so that the relation $\lembed{*}$ places some extra
properties on the inclusion $S\subset T$. Such a relation will be
called \emph{monotone} if $A\subset S\lembed{*} T\subset B$ implies
$A\lembed{*} B$. Various examples of monotone embedding relations will
be given in the next sections. In fact, several $UV$-properties can be
expressed in these terms by saying that a subset $S\subset Y$ is
$UV^*$ if every neighbourhood $U$ of $S$ contains a neighbourhood $V$
of $S$ such that $V\lembed{*} U$.\medskip

In what follows, all covers are assumed to be of \emph{nonempty}
sets. Also, in all statements in this section, $(Y,\rho)$ is a fixed
metric space and $\lembed{*}$ is a monotone embedding relation on the
subsets of $Y$. In this general setting, we shall say that a subset
$S\subset Y$ is \emph{uniformly $UV^*$} if for every $\varepsilon>0$
there exists $\delta=\delta(\varepsilon)>0$ such that
$\mathbf{O}_\delta(S)\lembed{*} \mathbf{O}_\varepsilon(S)$. Clearly, we
may always pick $\delta$ such that
$\delta(\varepsilon)\leq \varepsilon$, which is in good accord with
the requirement that
$\mathbf{O}_\delta(S)\subset \mathbf{O}_\varepsilon(S)$. \medskip

The main goal of this section is the following approximative
representation of $\rho$-continuous mappings with uniformly
$UV^*$-values.

\begin{theorem}
  \label{theorem-Near-sel-v1:1}
  Let $X$ be a paracompact space and $\varphi:X\sto Y$ be a
  $\rho$-continuous mapping such that each $\varphi(x)$, $x\in X$, is
  uniformly $UV^*$. Then for every continuous function
  $\varepsilon:X\to (0,+\infty)$ there exists a sequence
  $\mathscr{U}_n$, $n<\omega$, of open covers of $X$ and mappings
  $\Phi_n,\Psi_n:\mathscr{U}_n\sto Y$ such that for every
  $U\in \mathscr{U}_n$,
  \begin{enumerate}[itemsep=.5ex,label=\upshape{(\roman*)}]
  \item\label{item:Near-sel-v3:1}
    $\Phi_n(U)\lembed{*}\Psi_n(U)\subset
    \mathbf{O}_{\varepsilon(p)}(\varphi(p))$, whenever
    $p\in U$\textup{;}
  \item\label{item:Near-sel-v3:2} $\Psi_k(V)\subset \Phi_n(U)$,
    whenever $k>n$ and $V\in \mathscr{U}_k$ with
    $V\cap U\neq \emptyset$.
  \end{enumerate}
\end{theorem}

To prepare for the proof, we proceed with several observations.

\begin{proposition}
  \label{proposition-Near-sel-v1:1}
  Let $\varepsilon:X\to (0,+\infty)$ be a continuous function on a
  space $X$. Then there exists an open cover $\mathscr{V}$ of $X$ and
  a function $\delta:\mathscr{V}\to (0,+\infty)$ such that
  $\delta(V)\leq \varepsilon(p)$, for every $p\in V\in \mathscr{V}$.
\end{proposition}

\begin{proof}
  Since $\varepsilon$ is continuous, each point $x\in X$ is contained
  in an open set $V_x$ such that
  $\varepsilon(p)>\frac{\varepsilon(x)}2=\delta_x$, for every
  $p\in V_x$. Take $\mathscr{V}=\{V_x:x\in X\}$, and for each
  $V\in \mathscr{V}$ set $\delta(V)=\delta_x$, where $\delta_x$
  corresponds to some $x\in X$ with $V=V_{x}$.
\end{proof}

Let $\mathscr{F}(Y)$ be the collection of all nonempty closed subsets
of $Y$. A natural topology on $\mathscr{F}(Y)$ is the Hausdorff one
$\tau_{H(\rho)}$ generated by the Hausdorff distance $H(\rho)$
associated to $\rho$. Recall that the \emph{Hausdorff distance} is
defined by
\begin{align}
  \label{eq:Near-sel-v4:1}
  H(\rho)(S,T)&=\inf\big\{\varepsilon>0: S\subset
                \mathbf{O}_\varepsilon(T)\ \text{and}\ T\subset
                \mathbf{O}_\varepsilon(S)\big\}\\
              &=\sup\left\{\rho(S,y)+\rho(y,T): y\in S\cup
                T\right\},\qquad S,T\in \mathscr{F}(Y).\nonumber
\end{align}
The Hausdorff topology $\tau_{H(\rho)}$ is always metrizable even
though the distance $H(\rho)$ may assume infinite values. Moreover,
given any nonempty set $S\subset Y$, we have that
$\rho(y,S)=\rho\left(y,\overline{S}\right)$, for every $y\in Y$. So,
we may consider $H(\rho)(S,T)$ defined on all nonempty subsets of
$Y$. Then, according to \eqref{eq:Finite-C-v2:3} and
\eqref{eq:Near-sel-v4:1}, a mapping $\varphi:X\sto Y$ is
$\rho$-continuous if and only if it is continuous with respect to
$H(\rho)$ as a usual map, namely if for every $\varepsilon>0$, each
point $x\in X$ is contained in an open set $V$ such that
$H(\rho)(\varphi(p),\varphi(x))<\varepsilon$ for every $p\in V$.

\begin{proposition}
  \label{proposition-Finite-C-v4:1}
  Let $\varphi:X\sto Y$ be a $\rho$-continuous mapping on a space $X$,
  and $x\in X$ be such that $\varphi(x)$ is uniformly $UV^*$. Then for
  every $\varepsilon>0$ there exists $\delta_x\in (0,\varepsilon)$ and
  an open set $U_x\subset X$ containing $x$ such that for every
  $p,q\in U_x$,
  \[
    H(\rho)(\varphi(p),\varphi(q))<\frac{\delta_x}2\quad
    \text{and}\quad \mathbf{O}_{\delta_x}(\varphi(p))\lembed{*}
    \mathbf{O}_\varepsilon(\varphi(p)).
  \]
\end{proposition}

\begin{proof}
  Take $\delta_x\in \left(0,\frac\varepsilon4\right)$ with
  $\mathbf{O}_{2\delta_x}(\varphi(x)) \lembed{*}
  \mathbf{O}_{\frac\varepsilon2}(\varphi(x))$, and set
  \[
    U_x=\left\{p\in X:
      H(\rho)(\varphi(p),\varphi(x))<\frac{\delta_x}4\right\}.
  \]
  Then $U_x$ is an open set such that
  $H(\rho)(\varphi(p),\varphi(q))<\frac{\delta_x}2$, for every
  $p,q\in U_x$. Moreover, $p\in U_x$ implies
  $\mathbf{O}_{\delta_x}(\varphi(p))\lembed{*}
  \mathbf{O}_\varepsilon(\varphi(p))$ because $\lembed{*}$ is monotone
  and
  \[
    \mathbf{O}_{\delta_x}(\varphi(p))\subset
    \mathbf{O}_{2\delta_x}(\varphi(x)) \lembed{*}
    \mathbf{O}_{\frac\varepsilon2}(\varphi(x)) \subset
    \mathbf{O}_\varepsilon(\varphi(p)).\qedhere
  \]
\end{proof}

\begin{proposition}
  \label{proposition-Finite-C-v4:3}
  Let $X$ be a paracompact space and $\varphi:X\sto Y$ be a
  $\rho$-continuous mapping such that each $\varphi(x)$, $x\in X$, is
  uniformly $UV^*$. Suppose that $\mathscr{U}_1$ is a locally finite
  open cover of $X$ and $\delta_1:\mathscr{U}_1\to (0,+\infty)$. Then
  there exists an open locally finite refinement $\mathscr{U}_2$ of
  $\mathscr{U}_1$ and functions
  $\delta_2,\varepsilon_2:\mathscr{U}_2\to (0,+\infty)$ with
  $\delta_2\leq \varepsilon_2$, such that
  \begin{enumerate}[itemsep=.5ex,label=\upshape{(\roman*)}]
  \item
    $\mathbf{O}_{\delta_2(U)}(\varphi(p))\lembed{*}
    \mathbf{O}_{\varepsilon_2(U)}(\varphi(p))$, for every
    $p\in U\in \mathscr{U}_2$,
  \item $H(\rho)(\varphi(p),\varphi(q))<\frac{\delta_2(U)}2$, for
    every $p,q\in U\in \mathscr{U}_2$,
  \item $\varepsilon_2(U_2)\leq \frac{\delta_1(U_1)}3$, whenever
    $U_i\in \mathscr{U}_i$, $i=1,2$, with $U_1\cap U_2\neq \emptyset$.
  \end{enumerate}
\end{proposition}

\begin{proof}
  Since $\mathscr{U}_1$ is locally finite, the paracompact space $X$
  has a locally finite open cover $\mathscr{V}$ such that each family
  $\mathscr{U}_1^V=\left\{U\in \mathscr{U}_1: U\cap \overline{V}\neq
    \emptyset\right\}$, $V\in \mathscr{V}$, is finite. For
  convenience, set
  \begin{equation}
    \label{eq:Near-sel-vgg:1}
    \varepsilon(V)= \min_{U\in
      \mathscr{U}_1^V}\frac{\delta_1(U)}3>0,\quad \text{whenever $V\in
      \mathscr{V}$.}  
  \end{equation}
  Then by Proposition~\ref{proposition-Finite-C-v4:1}, for each
  $V\in \mathscr{V}$ there exists an open and locally finite in
  $\overline{V}$ cover $\mathscr{W}_V$ of the paracompact space
  $\overline{V}$ and a map
  $\delta_V:\mathscr{W}_V\to (0,\varepsilon(V))$ such that
  $\mathscr{W}_V$ refines $\mathscr{U}_1^{V}$ and for every
  $p,q\in W\in \mathscr{W}_V$,
  \begin{equation}
    \label{eq:Near-sel-vgg:3}
    H(\rho)(\varphi(p),\varphi(q))<\frac{\delta_V(W)}2\quad
    \text{and}\quad \mathbf{O}_{\delta_V(W)}(\varphi(p))\lembed{*}
    \mathbf{O}_{\varepsilon(V)}(\varphi(p)).
  \end{equation}
  
  Finally, we can take
  $\mathscr{U}_2=\left\{W\cap V: W\in \mathscr{W}_V\ \text{and}\ V\in
    \mathscr{V}\right\}$, which is a locally finite open cover of $X$
  because so is $\mathscr{V}$ and each $\mathscr{W}_V$,
  $V\in \mathscr{V}$, is locally finite in $\overline{V}$. To define
  the required maps
  $\delta_2,\varepsilon_2:\mathscr{U}_2\to (0,+\infty)$, for each
  $U\in \mathscr{U}_2$ fix elements $\lambda(U)\in \mathscr{V}$ and
  $\mu(U)\in \mathscr{W}_{\lambda(U)}$ with 
  $U=\mu(U)\cap \lambda(U)$, and set
  \begin{equation}
    \label{eq:Near-sel-vgg:4}
    \delta_2(U)=\delta_{\lambda(U)}(\mu(U))\quad \text{and}\quad
    \varepsilon_2(U)= \varepsilon(\lambda(U)),\ \text{
      $U\in \mathscr{U}_2$.}
  \end{equation}
  
  If $U_1\cap U_2\neq \emptyset$ for some $U_i\in \mathscr{U}_i$,
  $i=1,2$, then $U_1\cap \lambda(U_2)\neq \emptyset$ and, therefore,
  $U_1\in \mathscr{U}_1^{\lambda(U_2)}$. According to
  \eqref{eq:Near-sel-vgg:1}, we get that
  $\varepsilon_2(U_2)=\varepsilon(\lambda(U_2)) \leq
  \frac{\delta_1(U_1)}3$. The rest of the properties are evident from
  \eqref{eq:Near-sel-vgg:3} and \eqref{eq:Near-sel-vgg:4}. The proof
  is complete.
\end{proof}

We are now ready for the proof of Theorem~\ref{theorem-Near-sel-v1:1}.

\begin{proof}[Proof of Theorem~\ref{theorem-Near-sel-v1:1}]
  Let $\varphi$ and $\varepsilon$ be as in this theorem. Since $X$ is
  a paracompact space, by Proposition~\ref{proposition-Near-sel-v1:1},
  it has an open locally finite cover $\mathscr{V}$ and a function
  $\delta:\mathscr{V}\to (0,+\infty)$ such that
  \begin{equation}
    \label{eq:Near-sel-v1:1}
    \delta(V)\leq \varepsilon(p),\quad \text{for every
      $p\in V\in \mathscr{V}$.}
  \end{equation}
  Hence, by Proposition~\ref{proposition-Finite-C-v4:3} (applied with
  $\mathscr{U}_1=\mathscr{V}$ and $\delta_1=\delta$), $X$ has an open
  locally finite cover $\mathscr{U}_0$ which refines $\mathscr{V}$,
  and maps $\delta_0,\varepsilon_0: \mathscr{U}_0\to (0,+\infty)$ with
  $\delta_0\leq \varepsilon_0$, such that for every
  $p,q\in U\in \mathscr{U}_0$,
  \begin{equation}
    \label{eq:Near-sel-v1:2}
    \begin{cases}
      \mathbf{O}_{\delta_0(U)}(\varphi(p))\lembed{*}
      \mathbf{O}_{\varepsilon_0(U)}(\varphi(p)),\\[3pt]
      H(\rho)(\varphi(p),\varphi(q))<\frac{\delta_0(U)}2,\\
      \varepsilon_0(U)\leq \frac{\varepsilon(p)}2.\\
    \end{cases}
  \end{equation}
  To see the last property in \eqref{eq:Near-sel-v1:2}, take
  $p\in U\in \mathscr{U}_0$ and $V\in \mathscr{V}$ with $U\subset
  V$. Then by \eqref{eq:Near-sel-v1:1} and
  Proposition~\ref{proposition-Finite-C-v4:3},
  $\varepsilon_0(U)\leq \frac{\delta(V)}3\leq
  \frac{\varepsilon(p)}3\leq \frac{\varepsilon(p)}2$. \smallskip

  Using Proposition~\ref{proposition-Finite-C-v4:3}, the construction
  can be carried on by induction. Thus, there exists a sequence
  $\mathscr{U}_n$, $n<\omega$, of locally finite open covers of $X$
  and functions $\delta_n,\varepsilon_n:\mathscr{U}_n\to (0,+\infty)$
  with $\delta_n\leq \varepsilon_n$, such that each
  $\mathscr{U}_{n+1}$ refines $\mathscr{U}_n$ and the following
  properties hold for each $p,q\in U\in \mathscr{U}_{n}$:
  \begin{equation}
    \label{eq:Near-sel-v2:1}
    \begin{cases}
      \mathbf{O}_{\delta_n(U)}(\varphi(p))\lembed{*}
      \mathbf{O}_{\varepsilon_n(U)}(\varphi(p)),\\[3pt]
      H(\rho)(\varphi(p),\varphi(q))<\frac{\delta_n(U)}2,\\
      \varepsilon_{n+1}(V)\leq \frac{\delta_n(U)}3,\ \text{whenever
        $V\in \mathscr{U}_{n+1}$ with $V\cap U\neq \emptyset$.}
    \end{cases}
  \end{equation}

  To finish the proof, for each $n<\omega$ take a map
  $\pi_n:\mathscr{U}_n\to X$ with $\pi_n(U)\in U$,
  $U\in \mathscr{U}_n$. Next, define
  $\Phi_n,\Psi_n:\mathscr{U}_n\sto Y$ by
  \begin{equation}
    \label{eq:Near-sel-v5:1}
    \Phi_n(U)=\mathbf{O}_{\delta_n(U)}(\varphi(\pi_n(U)))\quad
    \text{and}\quad
    \Psi_n(U)=\mathbf{O}_{\varepsilon_n(U)}(\varphi(\pi_n(U))),\ U\in
    \mathscr{U}_n.
  \end{equation}

  To see \ref{item:Near-sel-v3:2}, take $k>n$, $U\in \mathscr{U}_n$
  and $V_k\in \mathscr{U}_{k}$ with $V_k\cap U\neq \emptyset$. Since
  each $\mathscr{U}_{i+1}$ refines $\mathscr{U}_i$, $i<\omega$, there
  are elements $V_i\in \mathscr{U}_i$, $n+1\leq i< k$, such that
  $V_{i+1}\subset V_{i}$ for every $n+1\leq i< k$.  Take a point
  $q\in V_k\cap U\subset V_{n+1}\cap U$. Since
  $\pi_k(V_k),q\in V_{n+1}$ and $q,\pi_n(U)\in U$, it follows from
  \eqref{eq:Near-sel-v2:1} that
  \begin{align*}
    H(\rho)(\varphi(\pi_{k}(V_k)),
    &\varphi(\pi_n(U)))\leq\\
    &\leq H(\rho)(\varphi(\pi_{k}(V_k)),\varphi(q))+
      H(\rho)(\varphi(q),\varphi(\pi_n(U)))\\
    &<
      \frac{\delta_{n+1}(V_{n+1})}2+\frac{\delta_n(U)}2
      \leq
      \frac12 \frac{\delta_{n}(U)}3+ \frac{\delta_{n}(U)}2=
      \frac{2\delta_{n}(U)}3.  
  \end{align*}
  For the same reason,
  $\varepsilon_k(V_k)\leq \dots \leq \varepsilon_{n+1}(V_{n+1})\leq
  \frac{\delta_n(U)}3$. Accordingly,
  \begin{align*}
    \Psi_{k}(V_k)=\mathbf{O}_{\varepsilon_{k}(V_k)}(\varphi(\pi_{k}(V_k)))
    &\subset
      \mathbf{O}_{\frac{\delta_n(U)}3}(\varphi(\pi_{k}(V_{k})))\\
    &\subset
      \mathbf{O}_{\frac{\delta_n(U)}3+
      \frac{2\delta_n(U)}3}(\varphi(\pi_n(U)))=\Phi_n(U).
  \end{align*}

  Finally, we show that \ref{item:Near-sel-v3:1} holds as well. The
  ``$\lembed{*}$''-embedding property in \ref{item:Near-sel-v3:1} is
  evident from \eqref{eq:Near-sel-v2:1} and
  \eqref{eq:Near-sel-v5:1}. As for the inclusion
  $\Psi_n(U)\subset \mathbf{O}_{\varepsilon(p)}(\varphi(p))$, where
  ${p\in U\in \mathscr{U}_n}$, it follows from the special case of
  $n=0$. Indeed, for $n>0$, take $U_0\in \mathscr{U}_0$ with
  $U\subset U_0$, which is possible because $\mathscr{U}_n$ refines
  $\mathscr{U}_0$. Then according to \ref{item:Near-sel-v3:2},
  $\Psi_n(U)\subset \Phi_0(U_0)$ because $U\cap U_0\neq
  \emptyset$. Thus, by \eqref{eq:Near-sel-v1:2} and
  \eqref{eq:Near-sel-v5:1}, we also have that
  \begin{align*}
    \Psi_n(U)\subset
    \Phi_0(U_0)&=
                 \mathbf{O}_{\delta_0(U_0)}(\varphi(\pi_0(U_0)))\\
               &\subset \mathbf{O}_{2\delta_0(U_0)}(\varphi(p))\subset
                 \mathbf{O}_{2\varepsilon_0(U_0)}(\varphi(p))\subset 
                 \mathbf{O}_{\varepsilon(p)}(\varphi(p)).\qedhere 
  \end{align*}
\end{proof}

\section{Oriented Simplicial Complexes} 
\label{sec:orient-simpl-compl}

As stated in the Introduction, we identify the vertices of a
simplicial complex $\Sigma$ with its $0$-skeleton $\Sigma^0$. A
simplicial complex $\Sigma$ is called \emph{oriented} if there exists
a partial order $\leq$ on its vertex set $\Sigma^0$ such that each
simplex of $\Sigma$ is linearly ordered with respect to
$\leq$. \medskip

The present section deals with two constructions of continuous maps on
oriented simplicial complexes. In the first construction, for a space
$Y$ and subsets ${S\subset T\subset Y}$, we write $S\lembed{\infty} T$
if $S$ is contractible in $T$.  Moreover, for mappings
$\Phi,\Psi:Z\sto Y$, we will use $\Phi\lembed{\infty}\Psi$ to express
that $\Phi(z)\lembed{\infty} \Psi(z)$, for every $z\in Z$.

\begin{theorem}
  \label{theorem-Near-sel-v7:1}
  Let $\Sigma$ be an oriented simplicial complex with respect to some
  partial order $\leq$ on $\Sigma^0$, and $\Phi,\Psi:\Sigma^0\sto Y$
  be mappings into a space $Y$ such that 
  $\Phi\lembed{\infty} \Psi$ and $\Psi(v)\subset \Phi(u)$,
  whenever $u,v\in \Sigma^0$ with $u<v$.  Then there exists a continuous
  map $h:|\Sigma|\to Y$ such that
  \begin{equation}
    \label{eq:Near-sel-v7:1}
    h(|\sigma|)\subset \Psi(\min \sigma),\quad\text{for every
      $\sigma\in \Sigma$.}
  \end{equation}
\end{theorem}

The proof of Theorem~\ref{theorem-Near-sel-v7:1} is based on the
following considerations. The \emph{cone} $Z*v$ over a space $Z$ with
a vertex $v$ is the quotient space of $Z\times[0,1]$ obtained by
identifying all points of $Z\times\{1\}$ into a single point $v$.  The
following representation of $Z*v$ will play an important role in
our proof.
\begin{equation}
  \label{eq:Near-sel-v7:2}
  Z*v=\big\{t v + (1-t) z: z\in Z\ \text{and}\ t\in [0,1]\big\}.
\end{equation}
This representation is unique for each point $q\in Z*v$ with
$q\neq v$, and is in good accord with the fact that
$|\sigma|=|\tau|*v$, whenever $\sigma$ is an abstract simplex,
$v\in \sigma$ and $\tau=\sigma\setminus\{v\}\neq \emptyset$ is the
face opposite to the vertex $v$.  Finally, for a space $T$ and
$S\subset T$, let us recall that a homotopy $H:S\times [0,1]\to T$ is
a contraction of $S$ into a point $p\in T$ if
\begin{equation}
  \label{eq:Near-sel-v7:4}
  H(y,0)=y\ \ \text{and}\ \ 
  H(y,1)=p,\quad \text{for every $y\in S$.}
\end{equation}

\begin{proposition}
  \label{proposition-Near-sel-v7:2}
  Let $H:S\times [0,1]\to T$ be a contraction of a subset $S\subset T$
  into a point $p\in T$ of a space $T$. Then for a space $Z$, each
  continuous map $g:Z\to S$ can be extended to a continuous map
  $h:Z*v\to T$ such that
  \begin{equation}
    \label{eq:Near-sel-v7:3}
    h\big(tv+(1-t)z\big)=H(g(z),t),\quad \text{for every $z\in Z$ and
      $t\in [0,1]$.}
  \end{equation}
\end{proposition}

\begin{proof}
  Define $h:Z*v\to T$ as in \eqref{eq:Near-sel-v7:3}, which is
  possible because $H(g(z),1)=p$ for every $z\in Z$, see
  \eqref{eq:Near-sel-v7:2} and \eqref{eq:Near-sel-v7:4}. This $h$ is
  as required because $H(g(z),0)=g(z)$, for every $z\in Z$. 
\end{proof}

We apply the above construction to show the special case of
Theorem~\ref{theorem-Near-sel-v7:1} in the setting of an oriented
simplex $\sigma$ by a linear order $\leq$ on its vertices. For
convenience, such a simplex will be denoted by $(\sigma,\leq)$.

\begin{proposition}
  \label{proposition-Near-sel-v8:1}
  Let $(\sigma,\leq)$ be an oriented simplex and
  $\Phi,\Psi:\sigma\sto Y$ be mappings into a space $Y$ such that
  $\Phi\lembed{\infty}\Psi$ and $\Psi(v)\subset \Phi(u)$, for every
  $u,v\in \sigma$ with $u<v$. For every $v\in \sigma$, fix a
  contraction $H_v:\Phi(v)\times[0,1]\to \Psi(v)$ into some point
  $p_v\in \Psi(v)$. Then there exists a unique continuous map
  $h:|\sigma|\to Y$ such that
  \begin{enumerate}[itemsep=.5ex,label=\upshape{(\roman*)}]
  \item\label{item:Near-sel-v8:1} $h(v)=p_v$, $v\in \sigma$,
  \item\label{item:Near-sel-v8:2} $h\big(tv+(1-t)z\big)=H_v(h(z),t)$,
    $z\in |\tau|$ and $t\in [0,1]$, whenever $\tau\subset \sigma$ is a
    face and $v\in \sigma$ with $v< u$ for every $u\in \tau$.
  \end{enumerate}
\end{proposition}

\begin{proof}
  Suppose that $\sigma=\{v_0,\dots, v_{n+1}\}$, where
  $v_0<v_1<\dots <v_{n+1}$, and consider the maximal chain
  $\tau_{n+1}\subset \tau_n\subset\dots \subset \tau_0=\sigma$ of
  faces of $\sigma$ defined by $\tau_{n+1}=\{v_{n+1}\}$ and
  $\tau_k=\tau_{k+1}\cup\{v_k\}$, $k\leq n$. We will construct the
  required map $h:|\sigma|\to Y$ inductively on the faces $|\tau_k|$,
  $k\leq n+1$, of this chain. To this end, set $S_k=\Phi(v_k)$,
  $T_k=\Psi(v_k)$, $H_k=H_{v_k}$ and $p_k=p_{v_k}$, $k\leq n+1$. Next,
  define $h_{n+1}:|\tau_{n+1}|\to T_{n+1}$ by
  $h_{n+1}(v_{n+1})=p_{n+1}$. Since $|\tau_n|=|\tau_{n+1}|*v_n$ and
  $T_{n+1}=\Psi(v_{n+1})\subset \Phi(v_n)=S_n$, by
  Proposition~\ref{proposition-Near-sel-v7:2}, $h_{n+1}$ has a
  continuous extension $h_n:|\tau_n|\to T_n$ such that
  \begin{align*}
    h_n\big(tv_n+(1-t)z\big)&=H_n(h_{n+1}(z),t)\\
                            &=H_n(h_n(z),t),\quad\text{$z\in
                              |\tau_{n+1}|$ and $t\in [0,1]$.} 
  \end{align*}
  It is also evident that $h_n(v_n)=H_n(p_{n+1},1)=p_n$. The
  construction can be carried on by a finite induction to get a
  continuous map $h=h_0:|\tau_0|=|\sigma|\to T_0$ satisfying
  \ref{item:Near-sel-v8:1} and such that for every
  $0\leq k\leq n$, 
  \begin{equation}
    \label{eq:Near-sel-v8:1}
    h\big(tv_k+(1-t)z\big)=H_{k}(h(z),t),\quad \text{$z\in
      |\tau_{k+1}|$ and $t\in [0,1]$.} 
  \end{equation}
  This $h$ is as required. Indeed, take any continuous map
  $g:|\sigma|\to Y$ satisfying \ref{item:Near-sel-v8:1} and
  \eqref{eq:Near-sel-v8:1}.  Then $g(v_{n+1})=h(v_{n+1})$ and by
  \eqref{eq:Near-sel-v8:1}, we get that
  ${g\uhr|\tau_n|=h\uhr|\tau_n|}$. Inductively, this implies that
  $g=g\uhr|\tau_0|=h\uhr |\tau_0|=h$.\medskip

  To see finally that \ref{item:Near-sel-v8:2} is equivalent to
  \eqref{eq:Near-sel-v8:1}, take a face $\tau\subset \sigma$ and a
  vertex $v\in \sigma$ with $v<u$, for every $u\in \tau$. Then $v=v_k$
  for some $0\leq k\leq n$, and $\tau\subset \tau_{k+1}$. If
  $z\in |\tau|\subset |\tau_{k+1}|$ and $t\in [0,1]$, it follows from
  \eqref{eq:Near-sel-v8:1} that
  \[
    h\big(tv+(1-t)z\big)=h\big(tv_k+(1-t)z\big)=
    H_k(h(z),t)=H_v(h(z),t).\qedhere
  \]
\end{proof}

\begin{proof}[Proof of Theorem \ref{theorem-Near-sel-v7:1}]
  Let $\Sigma$ be an oriented simplicial complex with respect to some
  partial order $\leq$ on $\Sigma^0$, and $\Phi,\Psi:\Sigma^0\sto Y$
  be as in Theorem~\ref{theorem-Near-sel-v7:1}. For each vertex
  ${v\in \Sigma^0}$, fix a contraction
  $H_v:\Phi(v)\times[0,1]\to \Psi(v)$ into some point
  ${p_v\in \Psi(v)}$. Since each simplex $\sigma\in \Sigma$ is
  oriented with respect to $\leq$, by
  Proposition~\ref{proposition-Near-sel-v8:1} applied with
  $\Phi\uhr \sigma$, $\Psi\uhr \sigma$ and the fixed contractions
  $H_v$, $v\in \sigma$, there exists a unique continuous map
  $h_\sigma:|\sigma|\to Y$ satisfying both properties
  \ref{item:Near-sel-v8:1} and \ref{item:Near-sel-v8:2} in that
  proposition. Suppose that $\sigma_1,\sigma_2\in \Sigma$ are
  simplices which have a common face $\tau=\sigma_1\cap
  \sigma_2$. Then the linear order $\leq$ on $\tau$ is the same in
  each one of the oriented simplices $(\sigma_i,\leq)$,
  $i=1,2$. Hence, by Proposition~\ref{proposition-Near-sel-v8:1}, we
  get that $h_{\sigma_1}\uhr |\tau|= h_\tau=h_{\sigma_2}\uhr
  |\tau|$. Accordingly, we may define a map $h:|\Sigma|\to Y$ by
  $h\uhr |\sigma|=h_\sigma$, $\sigma\in \Sigma$. This map $h$ is as
  required. Indeed, it is continuous because so is each $h_\sigma$,
  $\sigma\in \Sigma$. To see that $h$ also has the property in
  \eqref{eq:Near-sel-v7:1}, take a simplex $\sigma\in \Sigma$, and set
  $v=\min\sigma$. If $\sigma=\{v\}$, then by \ref{item:Near-sel-v8:1}
  of Proposition~\ref{proposition-Near-sel-v8:1},
  $h(|\sigma|)=h(v)=h_\sigma(v)=p_v\in \Psi(v)$. Otherwise, if
  $\tau=\sigma\setminus\{v\}\neq \emptyset$, it follows from
  \ref{item:Near-sel-v8:2} of
  Proposition~\ref{proposition-Near-sel-v8:1} that for every
  $z\in |\tau|$ and $t\in [0,1]$,
  \[
    h\big(tv+(1-t)z\big)= h_\sigma\big(tv
    +(1-t)z\big)=H_v(h_\sigma(z),t)\in \Psi(v).
  \]
  Accordingly, $h(|\sigma|)\subset \Psi(v)$ because
  $|\sigma|=|\tau|*v$. 
\end{proof}

Recall that the $k$-skeleton $\Sigma^k$, $k\geq 0$, of a simplicial
complex $\Sigma$ is the simplicial subcomplex
$\Sigma^{k}=\{\sigma\in \Sigma:\card(\sigma)\leq k+1\}$.  A simplicial
complex $\Sigma$ is called \emph{finite-dimensional} if
$\Sigma=\Sigma^k$ for some ${k\geq 0}$, and we say that $\Sigma$ is
\emph{$k$-dimensional} if $k\geq0$ is the smallest integer with this
property, i.e.\ for which $\Sigma=\Sigma^k$. If $\Sigma$ is an
$(n+1)$-dimensional simplicial complex which is oriented with respect
to some partial order $\leq$ on $\Sigma^0$, then we may consider the
chain of subcomplexes
$\Sigma_{n+1}\subset \Sigma_n\subset \dots \subset \Sigma_0=\Sigma$
defined by
\begin{equation}
  \label{eq:Near-sel-v9:2}
  \begin{cases}
    \Sigma_k=(\Sigma_k)^{n+1-k}, &\text{and}\\
    \min\sigma\notin (\Sigma_{k+1})^0, &\text{whenever
      $\sigma\in \Sigma_k\setminus \Sigma_{k+1}$.}
  \end{cases}
\end{equation}
Namely, let $V_0$ be the minimal elements of $\Sigma^0$, i.e.\ $V_0$
consists of all $v\in \Sigma^0$ such that $v=u$ whenever $u\in
\Sigma^0$ with $u\leq v$. Next, let $V_1$ be the minimal elements of
$\Sigma^0\setminus V_0$. Thus, inductively, for each $k\leq n$, we may
define the set $V_{k+1}$ which consists of the minimal elements of
$\Sigma^0\setminus (V_0\cup\dots\cup V_k)$. Since $\Sigma$ is
$(n+1)$-dimensional, each simplex of $\Sigma$ has at most $n+2$
elements. Therefore, $\Sigma^0= V_0\cup \dots \cup V_{n+1}$ and we
may now define the required subcomplexes by 
\begin{equation}
  \label{eq:Near-sel-vgg:2}
  \Sigma_k=\left\{\sigma\in \Sigma: \sigma\subset V_k\cup\dots \cup
    V_{n+1}\right\},\quad k\leq n+1.
\end{equation}

Our second construction deals with finite-dimensional oriented
simplicial complexes by relaxing the requirement in
Theorem~\ref{theorem-Near-sel-v7:1} that $\Phi\lembed{\infty}\Psi$. To
this end, for a space $Y$ and $S\subset T\subset Y$, we will write
that $S\lembed{\omega} T$ if for each finite-dimensional simplicial
complex $\Sigma$, any continuous map $g:|\Sigma|\to S$ is
null-homotopic in $T$. Just like before, for mappings
$\Phi,\Psi:Z\sto Y$, we will use $\Phi\lembed{\omega}\Psi$ to
designate that $\Phi(z)\lembed{\omega} \Psi(z)$, for every
$z\in Z$.\medskip

For a simplicial complex $\Sigma$ and a point $v$ with
$v\notin |\Sigma|$, the \emph{cone} on $\Sigma$ with a vertex $v$, is
the simplicial complex defined by
\begin{equation}
  \label{eq:Near-sel-v9:3}
  \Sigma*v=\Sigma\cup\left\{\sigma\cup \{v\}:
    \sigma\in\Sigma\right\}\cup \{v\}.
\end{equation}
According to \eqref{eq:Near-sel-v7:2}, $|\Sigma*v|$ is the cone
$|\Sigma|*v$ over the geometric realisation $|\Sigma|$ of $\Sigma$.
Moreover, for a space $T$ and $S\subset T$, we have that
$S\lembed{\omega} T$ iff for each finite-dimensional simplicial
complex $\Sigma$ and $v\notin |\Sigma|$, any continuous map
$g:|\Sigma|\to S$ can be extended to a continuous map
$h:|\Sigma|*v\to T$ on the cone $|\Sigma|*v=|\Sigma*v|$, compare with
Proposition \ref{proposition-Near-sel-v7:2}. Thus, we also have
the following theorem.

\begin{theorem}
  \label{theorem-Near-sel-v9:1}
  Let $\Sigma$ be a finite-dimensional simplicial complex which is
  oriented with respect to some partial order $\leq$ on $\Sigma^0$,
  and $\Phi,\Psi:\Sigma^0\sto Y$ be mappings into a space $Y$ such
  that $\Phi\lembed{\omega} \Psi$ and $\Psi(v)\subset \Phi(u)$,
  whenever $u,v\in \Sigma^0$ with $u<v$.  Then there exists a
  continuous map $h:|\Sigma|\to Y$ such that
  \begin{equation}
    \label{eq:Near-sel-v9:1}
    h(|\sigma|)\subset \Psi(\min \sigma),\quad\text{for every
      $\sigma\in \Sigma$.}
  \end{equation}
\end{theorem}

\begin{proof}
  Suppose that $\Sigma$ is $(n+1)$-dimensional, and consider the chain
  of subcomplexes
  $\Sigma_{n+1}\subset \Sigma_n\subset\dots \subset \Sigma_0=\Sigma$
  with the properties in \eqref{eq:Near-sel-v9:2}, see also
  \eqref{eq:Near-sel-vgg:2}. We are going to construct continuous maps
  $h_k:|\Sigma_k|\to Y$, $k\leq n+1$, satisfying
  \eqref{eq:Near-sel-v9:1} and such that each $h_k$ is an extension of
  $h_{k+1}$. Then we can take $h=h_0$ because $\Sigma_0=\Sigma$.  To
  this end, using that $\Sigma_{n+1}=(\Sigma_{n+1})^0$ is
  $0$-dimensional, we may define $h_{n+1}:\Sigma_{n+1}\to Y$ by
  $h_{n+1}(v)\in \Psi(v)$, for every $v\in \Sigma_{n+1}$. Evidently,
  this $h_{n+1}$ satisfies \eqref{eq:Near-sel-v9:1} with respect to
  any singleton (simplex) of $\Sigma_{n+1}$. \smallskip

  Suppose that $h_{k+1}:\big|\Sigma_{k+1}\big|\to Y$ is a continuous
  map satisfying \eqref{eq:Near-sel-v9:1}, and let us show how to
  extend it to a continuous map $h_{k}:\big|\Sigma_{k}\big|\to Y$
  preserving this property. So, take any vertex $v\in \Sigma^0$ with
  $\{v\}\in \Sigma_k\setminus \Sigma_{k+1}$, and consider the cone
  $\Omega_v* v\subset \Sigma_k$ associated to the subcomplex
  $\Omega_v=\big\{\sigma\in \Sigma_{k+1}: \sigma\cup\{v\}\in
  \Sigma_{k}\big\}$, see \eqref{eq:Near-sel-v9:3}. If
  $\Omega_v=\emptyset$, then $\Omega_v*v$ is the singleton $\{v\}$,
  and we define $h_v:|\Omega_v*v|\to Y$ by $h_v(v)\in
  \Psi(v)$. Suppose that $\Omega_v\neq \emptyset$, and take any
  simplex $\sigma\in \Omega_v\subset \Sigma_{k+1}$. According to
  \eqref{eq:Near-sel-v9:2}, $v<\min\sigma$ and, therefore,
  $\Psi(\min\sigma)\subset \Phi(v)$. Since $h_{k+1}$ satisfies
  \eqref{eq:Near-sel-v9:1}, this implies that
  $h_{k+1}(|\sigma|)\subset \Phi(v)$. Thus,
  $h_{k+1}(|\Omega_v|)\subset \Phi(v)$ and using that
  $\Phi(v)\lembed{\omega} \Psi(v)$, there exists a continuous
  extension $h_v:|\Omega_v*v|\to \Psi(v)$ of $h_{k+1}\uhr
  |\Omega_v|$. Evidently, $h_v$ satisfies \eqref{eq:Near-sel-v9:1}
  with respect to the simplicial complex $\Omega_v*v$.  Finally, we
  may define the required continuous extension $h_k:|\Sigma_k|\to Y$
  of $h_{k+1}$ by letting $h_k\uhr |\Sigma_{k+1}|=h_{k+1}$ and
  $h_k\uhr |\Omega_v*v|=h_v$, for every vertex $v\in \Sigma^0$ with
  $\{v\}\in \Sigma_k\setminus \Sigma_{k+1}$.
\end{proof}

\begin{remark}
  \label{remark-Near-sel-v10:1}
  The embedding relation $\lembed{\omega}$ in
  Theorem~\ref{theorem-Near-sel-v9:1} was used only with subcomplexes
  of $\Sigma$. Hence, this theorem remains valid if the property
  ``$\Phi\lembed{\omega} \Psi$'' is relaxed to the one that for each
  simplicial subcomplex $\Omega\subset \Sigma$ and $v\in \Sigma^0$,
  any continuous map $g:|\Omega|\to \Phi(v)$ is null-homotopic in
  $\Psi(v)$.
\end{remark}

\section{Approximate Selections and $C$-spaces}
\label{sec:appr-select-c}

If a paracompact space is a countable union of closed $C$-subspaces,
then it is itself a $C$-space \cite[Theorem 4.1]{gutev-valov:00}, see
also \cite[Theorem 2.7]{addis-gresham:78}. Accordingly, the space $X$
in Theorem~\ref{theorem-Selections-IDS:1} is a $C$-space. In this
section, we show that Theorem~\ref{theorem-Selections-IDS:1} remains
valid for all paracompact $C$-spaces. It is based on the relation
``$S\lembed{\infty}T$'' defined in the previous section, which is
clearly a monotone embedding relation. In its terms, a subset $S$ of a
metric space $(Y,\rho)$ is \emph{CE} iff it is \emph{uniformly
  $UV^\infty$} (see Section~\ref{sec:monot-embedd-relat}), namely if
for every $\varepsilon>0$ there exists $\delta\in (0,\varepsilon]$
such that
$\mathbf{O}_\delta(S)\lembed{\infty} \mathbf{O}_\varepsilon(S)$.

\begin{theorem}
  \label{theorem-Finite-C-v6:1}
  Let $X$ be a paracompact $C$-space, $(Y,\rho)$ be a metric space and
  $\varphi:X\sto Y$ be a $\rho$-continuous mapping such that each
  $\varphi(x)$, $x\in X$, is uniformly $UV^\infty$.  Then $\varphi$
  has a continuous $\varepsilon$-selection, for every continuous
  function ${\varepsilon:X\to (0,+\infty)}$.
\end{theorem}

\begin{proof}
  For a continuous function $\varepsilon:X\to (0,+\infty)$, take open
  covers $\mathscr{U}_n$, $n<\omega$, of $X$ and mappings
  $\Phi_n,\Psi_n:\mathscr{U}_n\sto Y$ as those in
  Theorem~\ref{theorem-Near-sel-v1:1} (applied with the relation
  $\lembed{\infty}$ and the function $\varepsilon$). Then, since $X$
  is a $C$-space, there exists a sequence $\mathscr{V}_{n}$,
  $n<\omega$, of pairwise disjoint families of nonempty open subsets
  of $X$ such that each $\mathscr{V}_n$, $n<\omega$, refines
  $\mathscr{U}_n$, and $\mathscr{V}=\bigcup_{n<\omega}\mathscr{V}_n$
  is a cover for $X$. Moreover, we assume that the families
  $\mathscr{V}_n$, $n<\omega$, have no common elements.\medskip

  Let $\mathscr{N}(\mathscr{V})$ be the \emph{nerve} of $\mathscr{V}$,
  i.e.\ the simplicial complex defined by
  \begin{equation}
    \label{eq:Near-sel-v10:1}
    \mathscr{N}(\mathscr{V})= \left\{\sigma\subset \mathscr{V}:\sigma\
      \text{is finite and}\ \bigcap\sigma\neq\emptyset\right\}.  
  \end{equation}
  Then there exists a natural orientation of
  $\mathscr{N}(\mathscr{V})$ generated by the families
  $\mathscr{V}_n$, $n<\omega$. Namely, define a partial order $\leq$
  on $\mathscr{V}$ by $V<W$, whenever $V\in \mathscr{V}_n$ and
  $W\in \mathscr{V}_k$ are such that $V\cap W\neq \emptyset$
  and $n<k$. Evidently, by \eqref{eq:Near-sel-v10:1}, this partial
  order makes $\mathscr{N}(\mathscr{V})$ an oriented simplicial
  complex.\smallskip

  Since each $\mathscr{V}_n$ refines $\mathscr{U}_n$, there are maps
  $r_n:\mathscr{V}_n\to \mathscr{U}_n$, $n<\omega$, such that
  $V\subset r_n(V)$, for every $V\in \mathscr{V}_n$. We may now define
  mappings $\Phi,\Psi:\mathscr{V}\sto Y$ by
  $\Phi\uhr \mathscr{V}_n=\Phi_n\circ r_n$ and
  $\Psi\uhr \mathscr{V}_n=\Psi_n\circ r_n$, $n<\omega$. According to
  Theorem~\ref{theorem-Near-sel-v1:1} and the definition of the
  partial order $\leq$ on $\mathscr{V}$, for every
  $V\in \mathscr{V}$ we have that 
  \begin{equation}
    \label{eq:Near-sel-v4:2}
    \begin{cases}
      \Phi(V)\lembed{\infty}\Psi(V)\subset
      \mathbf{O}_{\varepsilon(p)}(\varphi(p)),\
      \text{whenever  $p\in V$, and }\\
      \Psi(V)\subset \Phi(U),\ \text{whenever $U\in \mathscr{V}$ with
        $U<V$ and $U\cap V\neq \emptyset$.}
    \end{cases}
  \end{equation} 
  Thus, by Theorem~\ref{theorem-Near-sel-v7:1}, there exists a
  continuous map ${h:\big|\mathscr{N}(\mathscr{V})\big|\to Y}$
  satisfying \eqref{eq:Near-sel-v7:1} with respect to this orientation
  of $\mathscr{N}(\mathscr{V})$.  Finally, since $X$ is paracompact,
  the cover $\mathscr{V}$ has a canonical map
  $g:X\to |\mathscr{N}(\mathscr{V})|$.  In other words, see
  \cite[Proposition 2.5]{gutev:2018a}, $g$ is a continuous map such
  that $g(x)\in |\Sigma_\mathscr{V}(x)|$, $x\in X$, where
  $\Sigma_\mathscr{V}(x)=\big\{\sigma\in \mathscr{N}(\mathscr{V}):
  x\in \bigcap\sigma\big\}$. The composite map $f=h\circ g$
  \begin{center}
    \begin{tikzcd}
      &\lvert\mathscr{N}(\mathscr{V})\rvert \arrow[d, "h"]\\
      {X} \arrow[ur, "g"] \arrow[r, rightsquigarrow, "\varphi"] & Y
    \end{tikzcd}
  \end{center}
  is now a continuous $\varepsilon$-selection for $\varphi$. Indeed,
  take $x\in X$ and a simplex ${\sigma\in \Sigma_\mathscr{V}(x)}$ with
  $g(x)\in |\sigma|$. Then $x\in \min \sigma$ and by
  \eqref{eq:Near-sel-v7:1} and \eqref{eq:Near-sel-v4:2},
  \[
    f(x)=h(g(x))\in h(|\sigma|)\subset
    \Psi(\min \sigma)\subset
    \mathbf{O}_{\varepsilon(x)}(\varphi(x)).\qedhere
  \]
\end{proof}

\section{Approximate Selections and Finite $C$-spaces} 
\label{sec:appr-select-finite}

Finite $C$-spaces were defined by Borst \cite{MR2280911} for separable
metrizable spaces, subsequently the definition was extended by Valov
\cite{valov:00} for arbitrary spaces. For simplicity, we will consider
these spaces in the realm of normal spaces. In this setting, a
(normal) space $X$ is called a \emph{finite $C$-space} if for any
sequence $\{\mathscr{U} _n:n<\omega\}$ of finite open covers of $X$
there exists a finite sequence $\{\mathscr{V} _n:n\leq k\}$ of
pairwise disjoint open families in $X$ such that each $\mathscr{V} _n$
refines $\mathscr{U} _n$ and $\bigcup_{n\leq k}\mathscr{V} _n$ is a
cover of $X$. The following characterisation of paracompact finite
$C$-spaces was obtained Valov in \cite[Theorem 2.4]{valov:00}, it
brings the defining property of these spaces closer to that of
paracompact $C$-spaces.

\begin{theorem}[\cite{valov:00}]
  \label{theorem-Selections-IDS-v2:2}
  A paracompact space $X$ is a finite $C$-space if and only if for any
  sequence $\{\mathscr{U} _n:n<\omega\}$ of open covers of $X$ there
  exists a finite sequence $\{\mathscr{V} _n:n\leq k\}$ of discrete
  open families in $X$ such that each $\mathscr{V} _n$ refines
  $\mathscr{U} _n$ and $\bigcup_{n\leq k}\mathscr{V} _n$ is a cover of
  $X$.
\end{theorem}

Regarding the proper place of finite $C$-spaces, it was shown by Valov
in \cite[Proposition 2.2]{valov:00} that a space $X$ is a finite
$C$-space if and only if its \v{C}ech-Stone compactification $\beta X$
is a $C$-space. Accordingly, each compact $C$-space is a finite
$C$-space. On the other hand, there are $C$-spaces whose
\v{C}ech-Stone compactification is not a $C$-space, see e.g.\
\cite[Remark 3.7]{MR2080284}.\medskip

In this section, we extend Theorem~\ref{theorem-Selections-IDS-v4:1}
to all paracompact finite $C$-spaces. It is based on the relation
``$S\lembed{\omega} T$'' defined in
Section~\ref{sec:orient-simpl-compl}, which is another example of a
monotone embedding relation.  In terms of this relation, we shall say
that a subset $S$ of a metric space $(Y,\rho)$ is \emph{uniformly
  $UV^\omega$} if for each $\varepsilon>0$ there exists
$\delta\in (0,\varepsilon]$ such that
$\mathbf{O}_\delta(S)\lembed{\omega} \mathbf{O}_\varepsilon(S)$.

\begin{theorem}
  \label{theorem-Finite-C-v5:1}
  Let $X$ be a paracompact finite $C$-space, $(Y,\rho)$ be a metric
  space and $\varphi:X\sto Y$ be a $\rho$-continuous mapping such that
  each $\varphi(x)$, $x\in X$, is uniformly $UV^\omega$.
  Then $\varphi$ has a continuous $\varepsilon$-selection, for every
  continuous function ${\varepsilon:X\to (0,+\infty)}$.
\end{theorem}

\begin{proof}
  We proceed as in the proof of Theorem~\ref{theorem-Finite-C-v6:1},
  but now using Theorem~\ref{theorem-Near-sel-v9:1} instead of
  Theorem~\ref{theorem-Near-sel-v7:1}. Namely, for a continuous
  function $\varepsilon:X\to (0,+\infty)$, take open covers
  $\mathscr{U}_n$, $n<\omega$, of $X$ and mappings
  $\Phi_n,\Psi_n:\mathscr{U}_n\sto Y$ as those in
  Theorem~\ref{theorem-Near-sel-v1:1} (applied with the embedding
  relation $\lembed{\omega}$ and the function $\varepsilon$). Since
  $X$ is a paracompact finite $C$-space, by
  Theorem~\ref{theorem-Selections-IDS-v2:2}, there exists a finite
  sequence $\mathscr{V}_{n}$, $n\leq k$, of discrete families of
  nonempty open subsets of $X$ such that each $\mathscr{V}_n$,
  $n\leq k$, refines $\mathscr{U}_n$, and
  $\mathscr{V}=\bigcup_{n\leq k}\mathscr{V}_n$ is a cover for
  $X$. Next, precisely as in the proof
  of Theorem~\ref{theorem-Finite-C-v6:1}, we define a partial order on
  $\mathscr{V}$ and mappings $\Phi,\Psi:\mathscr{V}\sto Y$ such that
  for every $V\in \mathscr{V}$,
  \begin{equation}
    \label{eq:Near-sel-v10:2}
    \begin{cases}
      \Phi(V)\lembed{\omega}\Psi(V)\subset
      \mathbf{O}_{\varepsilon(p)}(\varphi(p)),\
      \text{whenever  $p\in V$, and }\\
      \Psi(V)\subset \Phi(U),\ \text{whenever $U\in \mathscr{V}$ with
        $U<V$ and $U\cap V\neq \emptyset$.}
    \end{cases}
  \end{equation}
  Thus, by Theorem~\ref{theorem-Near-sel-v9:1}, there exists a
  continuous map $h:|\mathscr{N}(\mathscr{V})|\to Y$ satisfying
  \eqref{eq:Near-sel-v9:1} with respect to the oriented simplicial
  complex $\mathscr{N}(\mathscr{V})$.\smallskip
  
  We may now finish the proof precisely as before. Namely, since $X$
  is paracompact, the cover $\mathscr{V}$ admits a continuous map
  $g:X\to |\mathscr{N}(\mathscr{V})|$ such that $g(x)\in |\sigma_x|$,
  $x\in X$, where
  $\sigma_x=\{V\in \mathscr{V}: x\in V\}\in \mathscr{N}(\mathscr{V})$,
  see \cite{dowker:47}. Then the composite map $f=h\circ g$ is a
  continuous $\varepsilon$-selection for $\varphi$. Indeed,
  $x\in \min \sigma_x$ and by \eqref{eq:Near-sel-v9:1} and
  \eqref{eq:Near-sel-v10:2}, we get that
  $f(x)=h(g(x))\in h(|\sigma_x|)\subset \Psi(\min \sigma_x)\subset
  \mathbf{O}_{\varepsilon(x)}(\varphi(x))$.
\end{proof}

\begin{remark}
  \label{remark-Near-sel-vgg:2}
  For an infinite cardinal $\tau$, we shall say that a subset $S$ of a
  metric space $(Y,\rho)$ is \emph{uniformly $\tau$-$UV^\omega$} if
  for each $\varepsilon>0$ there exists $\delta\in(0,\varepsilon]$
  such that for every finite-dimensional simplicial complex $\Sigma$
  with $\card(\Sigma^0)<\tau$, every continuous map
  $g:|\Sigma|\to \mathbf{O}_\delta(S)$ is null-homotopic in
  $\mathbf{O}_\varepsilon(S)$. Also, let us recall that the
  \emph{degree of compactness} of space $X$ is the least infinite
  cardinal $\kappa(X)$ such that every open cover of $X$ has an open
  refinement of cardinality less than ${\kappa(X)}$, see
  \cite{choban-mihaylova-nedev:08}. Using the same proof,
  Theorem~\ref{theorem-Finite-C-v5:1} remains valid if each
  $\varphi(x)$, $x\in X$, is only assumed to be uniformly
  $\kappa(X)$-$UV^\omega$. Indeed, if $\mathscr{V}_n$, $n\leq k$, are
  as in that proof, then ${\card(\mathscr{V})<\kappa(X)}$ because each
  family $\mathscr{V}_n$, $n\leq k$, is discrete. Hence, according to
  Remark~\ref{remark-Near-sel-v10:1}, we can still apply
  Theorem~\ref{theorem-Finite-C-v6:1} to get a continuous map
  $h:|\mathscr{N}(\mathscr{V})|\to Y$ satisfying
  \eqref{eq:Near-sel-v9:1} with respect to the oriented simplicial
  complex $\mathscr{N}(\mathscr{V})$. Having already the map $h$, we
  can finish the proof precisely as before. One benefit of this
  refined version of Theorem~\ref{theorem-Finite-C-v5:1} is that it
  implies Theorem~\ref{theorem-Selections-IDS-v4:1}. Namely, for a
  compact space $X$ and a metric space $(Y,\rho)$, a mapping
  $\varphi:X\sto Y$ has uniformly $\kappa(X)$-$UV^\omega$-values
  precisely when its values are as in
  Theorem~\ref{theorem-Selections-IDS-v4:1}.
\end{remark}

\section{Approximate Selections and Absolute Retracts}

Here, by an A(N)R we mean a metrizable space which is an
\emph{Absolute} (\emph{Neighbourhood}) \emph{Retract} for the
metrizable spaces.  A closed subset $Y\subset E$ of a metric space
$(E,\rho)$ is a \emph{uniform neighbourhood retract} of $E$ if to
every $\varepsilon > 0$ there corresponds $\eta(\varepsilon) > 0$ such
that there exists a retraction
$r:\mathbf{O}_{\eta(+\infty)}(Y)=
\bigcup_{\varepsilon>0}\mathbf{O}_{\eta(\varepsilon)}(Y)\to Y$ with
$\rho(z,r(z)) < \varepsilon$, whenever
$z\in \mathbf{O}_{\eta(\varepsilon)}(Y)$. If one can take
$\eta(+\infty) = +\infty$ (so that the domain of $r$ is always $E$),
then $Y$ is called a \emph{uniform retract} of $E$. A metric space
$(Y,\rho)$ is called a \emph{uniform} A(N)R if it is a uniform
(neighbourhood) retract of any metric space containing $(Y,\rho)$
isometrically as a closed subset. Uniform ANR's and AR's were
considered by Toru\'{n}czyk \cite{MR365471,MR445503} and Michael
\cite{michael:79}, see also Dugundji \cite{dugundji:58}. It should be
remarked that these uniform retracts are slightly different from those
considered by Isbell \cite{MR141074,isbell:64}. We now have the
following application of Theorem~\ref{theorem-Finite-C-v6:1}.

\begin{corollary}
  \label{corollary-Near-sel-vgg:1}
  Let $X$ be a paracompact $C$-space, $(Y,\rho)$ be a metric space
  which is a uniform ANR, and $\varphi:X\to \mathscr{F}(Y)$ be a
  $\rho$-continuous mapping such that each $\varphi(x)$, $x\in X$, is
  a uniform retract of $Y$.  Then $\varphi$ has a continuous
  $\varepsilon$-selection, for every continuous function
  ${\varepsilon:X\to (0,+\infty)}$.
\end{corollary}

\begin{proof}
  According to Theorem~\ref{theorem-Finite-C-v6:1}, it suffices to
  show that each $\varphi(x)$, $x\in X$, is uniformly $UV^\infty$. To
  this end, take a point $p\in X$, and let
  $r:\mathbf{O}_{\eta(+\infty)}(\varphi(p))\to \varphi(p)$ be the
  corresponding retraction, where $\eta(\varepsilon)$ is as in the
  definition of a uniform retract for this subset
  $\varphi(p)\subset Y$. Moreover, since $(Y,\rho)$ is a uniform ANR,
  to each $\varepsilon>0$ there exists $\gamma(\varepsilon)>0$ such
  that any two $\gamma(\varepsilon)$-close continuous maps from an
  arbitrary space to $Y$ are $\varepsilon$-homotopic, see
  e.g.\ \cite[Theorem 6.8.6]{MR3099433}. Finally, set
  $\delta(\varepsilon)= \eta(\gamma(\varepsilon))$, and let us show
  that this works. So, let
  $g=r\uhr \mathbf{O}_{\delta(\varepsilon)}(\varphi(p))$ and
  $f:\mathbf{O}_{\delta(\varepsilon)}(\varphi(p))\to
  \mathbf{O}_{\delta(\varepsilon)}(\varphi(p))$ be the identity of
  $\mathbf{O}_{\delta(\varepsilon)}(\varphi(p))$. Since
  $\delta(\varepsilon)=\eta(\gamma(\varepsilon))$, we have that
  $\rho(f(z),g(z))= \rho(z,r(z))< \gamma(\varepsilon)$, for every
  $z\in \mathbf{O}_{\delta(\varepsilon)}(\varphi(p))$.  Hence, $f$ and
  $g$ are $\gamma(\varepsilon)$-close and by the choice of
  $\gamma(\varepsilon)$, they are also $\varepsilon$-homotopic.  On
  the other hand, $\varphi(p)$ is contractible being an AR, see
  \cite[Theorem 12.4]{dugundji:58}. Therefore, $g$ is null-homotopic
  in $\varphi(p)$. Thus, the map $f$, i.e.\ the identity of
  $\mathbf{O}_{\varepsilon}(\varphi(p))$, is null-homotopic as well.
\end{proof}

In case the domain of $\varphi$ is a paracompact finite $C$-space, we
may relax both properties that each $\varphi(x)$, $x\in X$, is a
uniform ANR and a contractible set (being an AR). In fact, the
contractibility can be also relaxed in
Corollary~\ref{corollary-Near-sel-vgg:1}. Namely, for a subset
$S\subset T$ of a space $T$ and $k\geq 0$, we write that
$S\lembed{k} T$ if every continuous map of the $k$-sphere in $S$ can
be extended to a continuous map of the $(k+1)$-ball in $T$. A space
$S$ is called \emph{aspherical}, or $C^\omega$, if $S\lembed{k} S$,
for every $k\geq 0$.  The following property describing aspherical
spaces follows easily from the defining relation that $S\lembed{k}S$,
for all $k\geq 0$.

\begin{proposition}
  \label{proposition-Finite-C-v2:1}
  A space $S$ is $C^\omega$ if and only if for every simplicial
  complex $\Sigma$, each continuous map $g:|\Sigma|\to S$ is
  null-homotopic.
\end{proposition}

The importance of aspherical spaces is that a metrizable space is an
AR iff it is an aspherical ANR, \cite[Theorem
12.4]{dugundji:58}. \medskip

Evidently, a compact metric space is a uniform A(N)R iff it is an
A(N)R. Following Lefschetz \cite{MR0007094}, a compact metric space
$(Y,\rho)$ is called $LC^*$ if for every $\varepsilon>0$ there exists
$\eta(\varepsilon)>0$ such that if $\Sigma$ is a finite simplicial
complex and $K\subset \Sigma$ is a subcomplex containing the vertex
set $\Sigma^0$, then every continuous map $g:|K|\to Y$ with
$\diam(g(|K\cap \sigma|))<\eta(\varepsilon)$ for every
$\sigma\in \Sigma$, can be extended to a continuous map
$f:|\Sigma|\to Y$ such that $\diam(f(|\sigma|))<\varepsilon$,
$\sigma\in \Sigma$. It was shown by Lefschetz \cite[Theorem
6.6]{MR0007094} that a compact metric space $(Y,\rho)$ is a (uniform)
ANR if and only if it is $LC^*$, while $(Y,\rho)$ is an AR iff it is
both $LC^*$ and $C^\omega$. In \cite[Theorem 7.1]{michael:79}, Michael
extended this characterisation to arbitrary uniform ANR's by showing
that a metric space $(Y,\rho)$ is a uniform ANR if and only if it has
the above Lefschetz property for arbitrary simplicial complexes
$\Sigma$ rather than finite. Here, we are interested in another
extension of the $LC^*$ property to arbitrary subsets of metric
spaces. Namely, we shall say that a subset $S\subset Y$ of a metric
space $(Y,\rho)$ is \emph{uniformly $LC^*$} if for every
$\varepsilon>0$ there exists $\eta(\varepsilon)>0$ such that if
$\Sigma$ is a finite-dimensional simplicial complex and
$K\subset \Sigma$ is a subcomplex containing the vertex set
$\Sigma^0$, then each continuous map $g:|K|\to S$ with
$\diam(g(|K\cap \sigma|))<\eta(\varepsilon)$ for every
$\sigma\in \Sigma$, can be extended to a continuous map
$f:|\Sigma|\to S$ such that $\diam(f(|\sigma|))<\varepsilon$,
$\sigma\in \Sigma$. The following property of uniformly $LC^*$ sets
was essentially established by Michael in \cite[Lemma
11.1]{michael:56b}.

\begin{proposition}
  \label{proposition-Near-sel-vgg:1}
  Let $S\subset Y$ be a uniformly $LC^*$ subset of a metric space
  $(Y,\rho)$. Then for every $\varepsilon>0$ there exists
  $\gamma(\varepsilon)>0$ such that if $Z$ is a finite-dimensional
  paracompact space and $g:Z\to \mathbf{O}_{\gamma(\varepsilon)}(S)$
  is a continuous map, then there exists a continuous map $h:Z\to S$
  such that $h(z)\in \mathbf{O}_\varepsilon(g(z))$, for every
  $z\in Z$.
\end{proposition}

\begin{proof}
  We repeat the proof of \cite[Lemma 11.1]{michael:56b}. Namely, let
  $\gamma(\varepsilon)=\frac14\eta\left(\frac{2\varepsilon}3\right)$,
  where $\eta(\varepsilon)\leq \varepsilon$ is as in the definition of
  uniformly $LC^*$ of $S$. Take a finite-dimensional paracompact space
  $Z$ and a continuous map
  $g:Z\to \mathbf{O}_{\gamma(\varepsilon)}(S)$. Then there exists a
  locally finite open cover $\mathscr{U}$ of $Z$ such that its nerve
  $\mathscr{N}(\mathscr{U})$ is finite-dimensional (see \cite[Theorem
  3.5]{dowker:47}) and $\diam(g(U))<\gamma(\varepsilon)$, for every
  $U\in \mathscr{U}$. Next, identifying the $0$-skeleton of
  $\mathscr{N}(\mathscr{U})$ with $\mathscr{U}$, define a (continuous)
  map $\ell:\mathscr{U}\to S$ such that
  $\ell(U)\in \mathbf{O}_{\gamma(\varepsilon)}(g(U))$, for every
  $U\in \mathscr{U}$. Hence, we also have that
  \begin{equation}
    \label{eq:Near-sel-vgg:5}
    \rho(\ell(U),g(z))<2\gamma(\varepsilon)\leq
    \frac\varepsilon3,\quad \text{for every 
      $z\in U\in \mathscr{U}$.}
  \end{equation}
  Thus, for $U,V\in \mathscr{U}$ with $U\cap V\neq \emptyset$, we get
  that
  $\rho(\ell(U),\ell(V))<4\gamma(\varepsilon)=
  \eta\left(\frac{2\varepsilon}3\right)$. Since
  $\mathscr{N}(\mathscr{U})$ is finite-dimensional and
  $\eta\left(\frac{2\varepsilon}3\right)$ is as in the definition of
  uniformly $LC^*$ of $S$, the map $\ell$ can be extended to a
  continuous map $f:|\mathscr{N}(\mathscr{U})|\to S$ such that
  ${\diam(f(|\sigma|))<\frac{2\varepsilon}3}$, for every
  $\sigma\in \mathscr{N}(\mathscr{U})$. Finally, let
  $\pi:Z\to |\mathscr{N}(\mathscr{U})|$ be a canonical map and
  $h:Z\to S$ be the composite map $h=f\circ \pi$.  If $z\in Z$, then
  $z\in \bigcap \sigma$ for some $\sigma\in
  \mathscr{N}(\mathscr{U})$. Take any $U\in \sigma\subset
  |\sigma|$. Since $\pi(z),U\in |\sigma|$ and $f(U)=\ell(U)$, it
  follows from \eqref{eq:Near-sel-vgg:5} that
  \[
    \rho(h(z),g(z))\leq \rho(f(\pi(z)),f(U))+ \rho(\ell(U),g(z))
                     <\frac{2\varepsilon}3 +\frac{\varepsilon}3
                     =\varepsilon.\qedhere 
  \]
\end{proof}

We may now apply Theorem~\ref{theorem-Finite-C-v5:1} to get the
following further consequence.

\begin{corollary}
  \label{corollary-Near-sel-vgg:2}
  Let $(Y,\rho)$ be a metric space which is a uniform ANR, $X$ be a
  paracompact finite $C$-space and $\varphi:X\sto Y$ be a
  $\rho$-continuous mapping such that each $\varphi(x)$, $x\in X$, is
  both $C^\omega$ and uniformly $LC^*$. Then $\varphi$ has a
  continuous $\varepsilon$-selection, for every continuous function
  $\varepsilon:X\to (0,+\infty)$.
\end{corollary}

\begin{proof}
  According to Theorem~\ref{theorem-Finite-C-v5:1}, it suffices to
  show that each $\varphi(x)$, $x\in X$, is uniformly $UV^\omega$. To
  this end, take a point $p\in X$, and let $\gamma(\varepsilon)$ be as
  in Proposition~\ref{proposition-Near-sel-vgg:1}. Since $(Y,\rho)$ is
  a uniform ANR, as in the proof of
  Corollary~\ref{corollary-Near-sel-vgg:1}, to each $\varepsilon>0$
  there exists $\eta(\varepsilon)>0$ such that any two
  $\eta(\varepsilon)$-close continuous maps from an arbitrary space to
  $Y$ are $\varepsilon$-homotopic. Finally, take
  $\delta(\varepsilon)= \gamma(\eta(\varepsilon))$, and let us show
  that this works. So, let
  $g:|\Sigma|\to \mathbf{O}_{\delta(\varepsilon)}(\varphi(p))$ be a
  continuous map for some finite-dimensional simplicial complex
  $\Sigma$. Then by Proposition~\ref{proposition-Near-sel-vgg:1},
  there exists a continuous map ${h:|\Sigma|\to \varphi(p)}$ such that
  $\rho(h(z),g(z))<\eta(\varepsilon)$, for all $z\in
  |\Sigma|$. Accordingly, $g$ and $h$ are $\eta(\varepsilon)$-close
  and by the choice of $\eta(\varepsilon)$, the maps $g$ and $h$ are
  $\varepsilon$-homotopic.  On the other hand, by
  Proposition~\ref{proposition-Finite-C-v2:1},
  $h:|\Sigma|\to \varphi(p)$ is null-homotopic in $\varphi(p)$ because
  $\varphi(p)$ is aspherical. This implies that $g$ is null-homotopic
  in $\mathbf{O}_{\varepsilon}(\varphi(p))$, which shows that
  $\varphi(p)$ is uniformly $UV^\omega$.
\end{proof}

%%\bibliographystyle{amsplain-ab}
%%\bibliography{gutev}
\end{document}